\documentclass[a4paper,11pt]{amsart}

\usepackage{hyperref}
\usepackage{xcolor}
\hypersetup{
    colorlinks,
    linkcolor={red!50!black},
    citecolor={blue!50!black},
    urlcolor={blue!80!black}
}

\usepackage{microtype}
\usepackage{MnSymbol}

\makeatletter
\newcommand{\eqnum}{\refstepcounter{equation}\textup{\tagform@{\theequation}}}
\makeatother

\newtheorem{thm}[equation]{Theorem}
\newtheorem*{thm*}{Theorem}
\newtheorem{lem}[equation]{Lemma}
\newtheorem{cor}[equation]{Corollary}
\newtheorem{prop}[equation]{Proposition}

\newtheorem*{defthm*}{Definition/Theorem}

\theoremstyle{definition}

\newtheorem*{exam*}{Example}

\usepackage[utf8]{inputenc}

\usepackage{xspace}

\newcommand{\changelocaltocdepth}[1]{%
  \addtocontents{toc}{\protect\setcounter{tocdepth}{#1}}%
  \setcounter{tocdepth}{#1}}

\newcommand{\nc}{\newcommand}
\nc{\renc}{\renewcommand}
\nc{\ssec}{\subsection}
\nc{\sssec}{\subsubsection}
\nc{\on}{\operatorname}
\nc{\term}[1]{#1\xspace}

\usepackage{tikz}
\usetikzlibrary{matrix}
\usepackage{tikz-cd}

\tikzset{
  commutative diagrams/.cd,
  arrow style=tikz,
  diagrams={>=latex}}
\makeatletter
\tikzset{
  column sep/.code=\def\pgfmatrixcolumnsep{\pgf@matrix@xscale*(#1)},
  row sep/.code   =\def\pgfmatrixrowsep{\pgf@matrix@yscale*(#1)},
  matrix xscale/.code=%
    \pgfmathsetmacro\pgf@matrix@xscale{\pgf@matrix@xscale*(#1)},
  matrix yscale/.code=%
    \pgfmathsetmacro\pgf@matrix@yscale{\pgf@matrix@yscale*(#1)},
  matrix scale/.style={/tikz/matrix xscale={#1},/tikz/matrix yscale={#1}}}
\def\pgf@matrix@xscale{1}
\def\pgf@matrix@yscale{1}
\makeatother

\usepackage{enumitem}
\setlist[enumerate,1]{label={(\alph*)},itemsep=\parskip,leftmargin=0pt}
\newlist{thmlist}{enumerate}{1}
\setlist[thmlist,1]{label={\em(\roman*)},ref={\upshape{(\roman*)}},itemsep=\parskip,leftmargin=0pt}     
\newlist{remlist}{enumerate}{1}
\setlist[remlist,1]{label={(\roman*)},itemsep=\parskip,leftmargin=0pt}
\newlist{inlinelist}{enumerate*}{1}
\setlist[inlinelist,1]{label={(\alph*)}}

\nc{\sA}{\ensuremath{\mathcal{A}}\xspace}
\nc{\sB}{\ensuremath{\mathcal{B}}\xspace}
\nc{\sC}{\ensuremath{\mathcal{C}}\xspace}
\nc{\sD}{\ensuremath{\mathcal{D}}\xspace}
\nc{\sE}{\ensuremath{\mathcal{E}}\xspace}
\nc{\sF}{\ensuremath{\mathcal{F}}\xspace}
\nc{\sG}{\ensuremath{\mathcal{G}}\xspace}
\nc{\sH}{\ensuremath{\mathcal{H}}\xspace}
\nc{\sI}{\ensuremath{\mathcal{I}}\xspace}
\nc{\sJ}{\ensuremath{\mathcal{J}}\xspace}
\nc{\sK}{\ensuremath{\mathcal{K}}\xspace}
\nc{\sL}{\ensuremath{\mathcal{L}}\xspace}
\nc{\sM}{\ensuremath{\mathcal{M}}\xspace}
\nc{\sN}{\ensuremath{\mathcal{N}}\xspace}
\nc{\sO}{\ensuremath{\mathcal{O}}\xspace}
\nc{\sP}{\ensuremath{\mathcal{P}}\xspace}
\nc{\sQ}{\ensuremath{\mathcal{Q}}\xspace}
\nc{\sR}{\ensuremath{\mathcal{R}}\xspace}
\nc{\sS}{\ensuremath{\mathcal{S}}\xspace}
\nc{\sT}{\ensuremath{\mathcal{T}}\xspace}
\nc{\sU}{\ensuremath{\mathcal{U}}\xspace}
\nc{\sV}{\ensuremath{\mathcal{V}}\xspace}
\nc{\sW}{\ensuremath{\mathcal{W}}\xspace}
\nc{\sX}{\ensuremath{\mathcal{X}}\xspace}
\nc{\sY}{\ensuremath{\mathcal{Y}}\xspace}
\nc{\sZ}{\ensuremath{\mathcal{Z}}\xspace}

\nc{\bA}{\ensuremath{\mathbf{A}}\xspace}
\nc{\bB}{\ensuremath{\mathbf{B}}\xspace}
\nc{\bC}{\ensuremath{\mathbf{C}}\xspace}
\nc{\bD}{\ensuremath{\mathbf{D}}\xspace}
\nc{\bE}{\ensuremath{\mathbf{E}}\xspace}
\nc{\bF}{\ensuremath{\mathbf{F}}\xspace}
\nc{\bG}{\ensuremath{\mathbf{G}}\xspace}
\nc{\bH}{\ensuremath{\mathbf{H}}\xspace}
\nc{\bI}{\ensuremath{\mathbf{I}}\xspace}
\nc{\bJ}{\ensuremath{\mathbf{J}}\xspace}
\nc{\bK}{\ensuremath{\mathbf{K}}\xspace}
\nc{\bL}{\ensuremath{\mathbf{L}}\xspace}
\nc{\bM}{\ensuremath{\mathbf{M}}\xspace}
\nc{\bN}{\ensuremath{\mathbf{N}}\xspace}
\nc{\bO}{\ensuremath{\mathbf{O}}\xspace}
\nc{\bP}{\ensuremath{\mathbf{P}}\xspace}
\nc{\bQ}{\ensuremath{\mathbf{Q}}\xspace}
\nc{\bR}{\ensuremath{\mathbf{R}}\xspace}
\nc{\bS}{\ensuremath{\mathbf{S}}\xspace}
\nc{\bT}{\ensuremath{\mathbf{T}}\xspace}
\nc{\bU}{\ensuremath{\mathbf{U}}\xspace}
\nc{\bV}{\ensuremath{\mathbf{V}}\xspace}
\nc{\bW}{\ensuremath{\mathbf{W}}\xspace}
\nc{\bX}{\ensuremath{\mathbf{X}}\xspace}
\nc{\bY}{\ensuremath{\mathbf{Y}}\xspace}
\nc{\bZ}{\ensuremath{\mathbf{Z}}\xspace}

\nc{\bbA}{\ensuremath{\mathbb{A}}\xspace}
\nc{\bbB}{\ensuremath{\mathbb{B}}\xspace}
\nc{\bbC}{\ensuremath{\mathbb{C}}\xspace}
\nc{\bbD}{\ensuremath{\mathbb{D}}\xspace}
\nc{\bbE}{\ensuremath{\mathbb{E}}\xspace}
\nc{\bbF}{\ensuremath{\mathbb{F}}\xspace}
\nc{\bbG}{\ensuremath{\mathbb{G}}\xspace}
\nc{\bbH}{\ensuremath{\mathbb{H}}\xspace}
\nc{\bbI}{\ensuremath{\mathbb{I}}\xspace}
\nc{\bbJ}{\ensuremath{\mathbb{J}}\xspace}
\nc{\bbK}{\ensuremath{\mathbb{K}}\xspace}
\nc{\bbL}{\ensuremath{\mathbb{L}}\xspace}
\nc{\bbM}{\ensuremath{\mathbb{M}}\xspace}
\nc{\bbN}{\ensuremath{\mathbb{N}}\xspace}
\nc{\bbO}{\ensuremath{\mathbb{O}}\xspace}
\nc{\bbP}{\ensuremath{\mathbb{P}}\xspace}
\nc{\bbQ}{\ensuremath{\mathbb{Q}}\xspace}
\nc{\bbR}{\ensuremath{\mathbb{R}}\xspace}
\nc{\bbS}{\ensuremath{\mathbb{S}}\xspace}
\nc{\bbT}{\ensuremath{\mathbb{T}}\xspace}
\nc{\bbU}{\ensuremath{\mathbb{U}}\xspace}
\nc{\bbV}{\ensuremath{\mathbb{V}}\xspace}
\nc{\bbW}{\ensuremath{\mathbb{W}}\xspace}
\nc{\bbX}{\ensuremath{\mathbb{X}}\xspace}
\nc{\bbY}{\ensuremath{\mathbb{Y}}\xspace}
\nc{\bbZ}{\ensuremath{\mathbb{Z}}\xspace}

\DeclareMathSymbol{A}{\mathalpha}{operators}{`A}
\DeclareMathSymbol{B}{\mathalpha}{operators}{`B}
\DeclareMathSymbol{C}{\mathalpha}{operators}{`C}
\DeclareMathSymbol{D}{\mathalpha}{operators}{`D}
\DeclareMathSymbol{E}{\mathalpha}{operators}{`E}
\DeclareMathSymbol{F}{\mathalpha}{operators}{`F}
\DeclareMathSymbol{G}{\mathalpha}{operators}{`G}
\DeclareMathSymbol{H}{\mathalpha}{operators}{`H}
\DeclareMathSymbol{I}{\mathalpha}{operators}{`I}
\DeclareMathSymbol{J}{\mathalpha}{operators}{`J}
\DeclareMathSymbol{K}{\mathalpha}{operators}{`K}
\DeclareMathSymbol{L}{\mathalpha}{operators}{`L}
\DeclareMathSymbol{M}{\mathalpha}{operators}{`M}
\DeclareMathSymbol{N}{\mathalpha}{operators}{`N}
\DeclareMathSymbol{O}{\mathalpha}{operators}{`O}
\DeclareMathSymbol{P}{\mathalpha}{operators}{`P}
\DeclareMathSymbol{Q}{\mathalpha}{operators}{`Q}
\DeclareMathSymbol{R}{\mathalpha}{operators}{`R}
\DeclareMathSymbol{S}{\mathalpha}{operators}{`S}
\DeclareMathSymbol{T}{\mathalpha}{operators}{`T}
\DeclareMathSymbol{U}{\mathalpha}{operators}{`U}
\DeclareMathSymbol{V}{\mathalpha}{operators}{`V}
\DeclareMathSymbol{W}{\mathalpha}{operators}{`W}
\DeclareMathSymbol{X}{\mathalpha}{operators}{`X}
\DeclareMathSymbol{Y}{\mathalpha}{operators}{`Y}
\DeclareMathSymbol{Z}{\mathalpha}{operators}{`Z}

\nc{\mrm}[1]{\ensuremath{\mathrm{#1}}\xspace}
\nc{\mit}[1]{\ensuremath{\mathit{#1}}\xspace}
\nc{\mbf}[1]{\ensuremath{\mathbf{#1}}\xspace}
\nc{\mcal}[1]{\ensuremath{\mathcal{#1}}\xspace}
\nc{\msc}[1]{\ensuremath{\mathscr{#1}}\xspace}

\renc{\ge}{\geqslant}
\renc{\le}{\leqslant}

\nc{\id}{\mathrm{id}}

\DeclareMathOperator{\Hom}{\on{Hom}}
\nc{\uHom}{\underline{\smash{\Hom}}}
\DeclareMathOperator{\Maps}{\on{Maps}}

\DeclareMathOperator{\End}{\on{End}}

\nc{\uEnd}{\underline{\smash{\End}}}

\nc{\colim}{\varinjlim}
\renc{\lim}{\varprojlim}
\nc{\Cofib}{\on{Cofib}}
\nc{\Fib}{\on{Fib}}
\nc{\initial}{\varnothing}
\nc{\op}{\mathrm{op}}

\DeclareMathOperator*{\fibprod}{\times}

\renc{\setminus}{\smallsetminus}

\usepackage{mathtools}
\newcommand{\thmref}[1]{Theorem~\ref{#1}}

\newcommand{\propref}[1]{Proposition~\ref{#1}}
\newcommand{\corref}[1]{Corollary~\ref{#1}}

\renewcommand{\eqref}[1]{(\ref{#1})}

\nc{\A}{\bA}
\renc{\P}{\bP}
\nc{\Spec}{\on{Spec}}
\nc{\Qcoh}{\on{Qcoh}}
\nc{\Perf}{\on{Perf}}
\nc{\Sch}{\mrm{Sch}}
\nc{\Sm}{\mrm{Sm}}
\nc{\Bl}{\on{Bl}}
\renc{\H}{\mbf{H}}
\nc{\SH}{\mbf{SH}}
\nc{\uH}{\underline{\H}}
\nc{\uSH}{\underline{\SH}}
\nc{\Nis}{{\mrm{Nis}}}
\nc{\cdh}{{\mrm{cdh}}}
\nc{\h}{\mrm{h}}
\renc{\L}{\mrm{L}}
\nc{\T}{{\mbf{T}}}
\nc{\un}{{\mbf{1}}}
\nc{\V}{{\bV}}
\nc{\Spc}{\mrm{Spc}}

\nc{\rightrightrightarrows}
  {\mathrel{\substack{\textstyle\rightarrow\\[-0.6ex]
   \textstyle\rightarrow \\[-0.6ex]
   \textstyle\rightarrow}}}
\nc{\rightrightrightrightarrows}
  {\mathrel{\substack{\textstyle\rightarrow\\[-0.6ex]
   \textstyle\rightarrow \\[-0.6ex]
   \textstyle\rightarrow \\[-0.6ex]
   \textstyle\rightarrow}}}

\nc{\inftyCat}{\term{$\infty$-category}}
\nc{\inftyCats}{\term{$\infty$-categories}}

\title{The cdh-local~motivic~homotopy category\vspace{-2mm}}
\author{Adeel~A.~Khan\vspace{-1mm}}
\date{2019-10-08}

\makeatletter
\def\l@subsection{\@tocline{2}{0pt}{4pc}{6pc}{}}
\makeatother

\begin{document}

\begin{abstract}
  We construct a cdh-local motivic homotopy category $\SH_\cdh(S)$ over an arbitrary base scheme $S$, and show that there is a canonical equivalence $\SH_\cdh(S) \simeq \SH(S)$.
  We learned this result from D.-C.~Cisinski.
  \vspace{-5mm}
\end{abstract}

\maketitle


\parskip 0.2cm
\thispagestyle{empty}


\changelocaltocdepth{1}
\section{Sm-fibred and Sch-fibred motivic spectra}

  Let $\sF \in \SH(S)$ be a motivic spectrum over a scheme $S$.
  Recall that $\sF$ defines a cohomology theory on smooth $S$-schemes by the formula
    \begin{equation*}
      \Gamma(X, \sF) = \Maps_{\SH(S)}(\Sigma^\infty_\T(X_+), \sF).
    \end{equation*}
  In terms of the six operations, we can write $\Sigma^\infty_\T(X_+) \simeq f_\sharp f^*(\un_S)$, where $f : X \to S$ is the structural morphism and $\un_S \in \SH(S)$ is the monoidal unit, and thus by adjunction
    \begin{equation*}
      \Gamma(X, \sF) \simeq \Maps_{\SH(S)}(\un_S, f_*f^*(\sF)).
    \end{equation*}
  Note that the right-hand side makes sense even when $X$ is not smooth.
  Thus the cohomology theory $X \mapsto \Gamma(X, \sF)$ naturally extends to arbitrary $S$-schemes.
  The language of ``$\Sch$-fibred motivic spectra'' gives a more concrete way to describe this extension that doesn't use the six operations (see \propref{prop:cohomology}).

\ssec{}
  Let $S$ be a qcqs scheme\footnote{One can replace ``scheme'' by ``algebraic space'' throughout the note.}.
  A \emph{$\Sch$-fibred space} over $S$ is a presheaf of spaces on the category $\Sch_{/S}$ of $S$-schemes of finite presentation.
  We say that a $\Sch$-fibred space $\sF$ is \emph{$\A^1$-invariant} if the canonical map $\Gamma(X,\sF) \to \Gamma(X\times\A^1,\sF)$ is invertible for every $X\in\Sch_{/S}$.
  It is \emph{Nisnevich-local} if it satisfies \v{C}ech descent with respect to the topology generated by Nisnevich squares in $\Sch_{/S}$.
  By a theorem of Voevodsky \cite[Thm.~3.2.5]{AsokHoyoisWendt}, a $\Sch$-fibred space is Nisnevich-local iff it is \emph{reduced}, i.e., $\Gamma(\initial,\sF)$ is contractible, and if it sends every Nisnevich square $Q$ to a cartesian square of spaces $\Gamma(Q, \sF)$.

\ssec{}
  A \emph{$\Sm$-fibred space} over $S$ is a presheaf of spaces on the category $\Sm_{/S}$ of \emph{smooth} $S$-schemes of finite presentation.
  We say a $\Sm$-fibred space $\sF$ is \emph{$\A^1$-local} if the canonical map $\Gamma(X,\sF) \to \Gamma(X\times\A^1,\sF)$ is invertible for every $X\in\Sm_{/S}$.
  It is \emph{Nisnevich-local} if it satisfies \v{C}ech descent with respect to the topology generated by Nisnevich squares in $\Sm_{/S}$.

\ssec{}
  Let $\uH(S)$ denote the \inftyCat of $\A^1$-local Nisnevich-local $\Sch$-fibred spaces over $S$.
  This is the left Bousfield localization of the \inftyCat of $\Sch$-fibred spaces at the class of ($\A^1,\Nis$)-\emph{local equivalences}.
  By \cite[Prop.~5.5.4.15(4)]{HTT} the latter class is the strongly saturated closure of the class of $\A^1$-projections and \v{C}ech nerves of Nisnevich coverings.
  We denote the localization functor by $\L_{\A^1,\Nis}$, or sometimes simply $\L$ when there is no risk of confusion.
  Similarly for the \inftyCat $\H(S)$ of $\A^1$-local Nisnevich-local $\Sm$-fibred spaces over $S$.

\ssec{}
  Denote by $\iota : \Sm_{/S} \hookrightarrow \Sch_{/S}$ the inclusion functor.
  Restriction of presheaves along $\iota$, which we denote $\sF \mapsto \iota^*(\sF)$, admits a fully faithful left adjoint $\sF \mapsto \iota_!(\sF)$ given by left Kan extension.
  The latter commutes with colimits and satisfies $\iota_! \h_S(X) = \h_S(X)$ for every $X \in \Sm_{/S}$, where $\h_S(-)$ denotes the Yoneda embedding.

  \begin{lem}\label{lem:iota}
    Both functors $\iota^*$ and $\iota_!$ preserve Nisnevich-local and $\A^1$-local equivalences.
  \end{lem}

  \begin{proof}
    Since $\iota$ preserves $\A^1$-projections, it follows that $\iota_!$ preserves $\A^1$-local equivalences.
    Since it is \emph{continuous} with respect to the Nisnevich topology (i.e., preserves Nisnevich coverings), it follows also that $\iota_!$ preserves Nisnevich-local equivalences.

    Note that $\iota$ is also \emph{cocontinuous} with respect to the Nisnevich topology, in the sense of \cite[Exp.~III]{SGA4}: that is, any Nisnevich square in $\Sch_{/S}$ over a smooth scheme $X \in \Sm_{/S}$ actually lies in $\Sm_{/S}$.
    This implies by \cite[Exp.~III, Prop.~2.2]{SGA4} that $\iota^*$ also preserves Nisnevich-local equivalences.
    The proof that $\iota^*$ preserves $\A^1$-local equivalences is similar, but let's spell it out.
    It suffices to show that, for any $X\in\Sch_{/S}$, the canonical morphism
      \begin{equation*}
        \iota^*\h_S(X \times \A^1) \to \iota^*\h_S(X)
      \end{equation*}
    is an $\A^1$-local equivalence of $\Sm$-fibred spaces.
    By universality of colimits it suffices to show that, for any $Y\in\Sm_{/S}$ and any morphism $\varphi : \h_S(Y) \to \iota^*\h_S(X)$ (corresponding to a morphism $Y \to X$ in $\Sch_{/S}$), the base change
      \begin{equation*}
        \iota^*\h_S(X \times \A^1) \fibprod_{\iota^*\h_S(X)} \h_S(Y) \to \h_S(Y)
      \end{equation*}
    is an $\A^1$-local equivalence.
    Since the morphism $\varphi$ factors as $\h_S(Y) \simeq \iota^*\h_S(Y) \to \iota^*\h_S(X)$, the morphism in question is identified with the morphism
      \begin{equation*}
        \iota^*\h_S(X \times \A^1) \fibprod_{\iota^*\h_S(X)} \iota^*\h_S(Y) \to \iota^*\h_S(Y),
      \end{equation*}
    which itself is identified with the canonical morphism $\h_S(Y \times \A^1) \to \h_S(Y)$, because $\iota^*$ and $\h_S$ commute with limits and because $\iota^*\iota_! = \id$.
    But this is an $\A^1$-local equivalence.
  \end{proof}

  \begin{prop}\label{prop:H into uH}
    The functor $\sF \mapsto \L_{\A^1,\Nis}\iota_!(\sF)$ defines a fully faithful embedding $\H(S) \hookrightarrow \uH(S)$.
    Its essential image is generated under small colimits by $(\A^1,\Nis)$-localizations of representables $\h_S(X)$, for $X \in \Sm_{/S}$.
  \end{prop}
  \begin{proof}
    The fact that $\iota^*$ preserves $(\A^1,\Nis)$-equivalences implies that its right adjoint $\iota_*$ (right Kan extension) preserves $\A^1$-invariant Nisnevich-local spaces, and restricts to a functor $\iota_* : \H(S) \to \uH(S)$.
    Similarly, the fact that $\iota_!$ preserves $(\A^1,\Nis)$-equivalences implies that its right adjoint $\iota^*$ restricts to a functor $\iota^* : \uH(S) \to \H(S)$, right adjoint to $\L\iota_!$.
    Now $\iota^* : \uH(S) \to \H(S)$ is a functor whose right adjoint is fully faithful, so by abstract nonsense its left adjoint $\L\iota_!$ is also fully faithful.
  \end{proof}

  \ssec{}
    We now consider the stable versions of our categories.
    Let $\T$ denote the Thom space $\A^1_S/(\A^1_S \setminus S)$ of the trivial line bundle over $S$.
    We let $\uSH(S)$, denote the \inftyCat of $\T$-spectra in the symmetric monoidal \inftyCat $\uH(S)_\bullet$; we refer to these simply as \emph{$\Sch$-fibred motivic spectra}.
    By construction, $\uSH(S)$ is generated under small colimits by objects of the form $\Omega^n_\T\Sigma^\infty_\T (X_+)$, for $X \in \Sch_{/S}$ and $n\ge0$.

    Note that $\uSH(S)$ can alternatively be realized as the left Bousfield localization of the \inftyCat of $\T$-spectra in pointed $\Sch$-fibred spaces.
    For example, a $\T$-spectrum $\sF$ of pointed $\Sch$-fibred spaces is $\A^1$-invariant or Nisnevich-local iff for every integer $n\ge 0$, the pointed $\Sch$-fibred space $\Omega^\infty_\T\Sigma_\T^n(\sF)$ has the respective property.

    Similarly, we have the \inftyCat $\SH(S)$ of \emph{$\Sm$-fibred motivic spectra}, i.e., $\T$-spectra in the symmetric monoidal \inftyCat $\H(S)_\bullet$.
    It is generated under small colimits by objects of the form $\Omega^n_\T\Sigma^\infty_\T (X_+)$, for $X \in \Sm_{/S}$ and $n\ge0$.

  \ssec{}
    Just like in the unstable case (\propref{prop:H into uH}), $\SH(S)$ can be described as a full subcategory of $\uSH(S)$.
    Since $\iota^* : \uH(S) \to \H(S)$ commutes with $\Omega_\T$, it admits a unique extension $\iota^* : \uSH(S) \to \SH(S)$ that commutes with $\Omega^\infty_\T$.
    Its fully faithful right adjoint $\iota_* : \H(S) \to \uH(S)$ similarly extends uniquely to a fully faithful functor $\iota_* : \SH(S) \to \uSH(S)$ that commutes with $\Omega^\infty_\T$.
    The left adjoint $\L\iota_!$ also extends uniquely to a functor
      \begin{equation}\label{eq:SH into uSH}
        \L\iota_! : \SH(S) \to \uSH(S)
      \end{equation}
    that commutes with $\Sigma^\infty_\T$.

    \begin{prop}\label{prop:SH into uSH}
      The functor \eqref{eq:SH into uSH} is fully faithful, with essential image generated under small colimits by objects of the form $\Omega^n_\T\Sigma^\infty_\T (X_+)$, for $X \in \Sm_{/S}$ and $n\ge0$.
    \end{prop}
    \begin{proof}
      Same trick as before: fully faithfulness of $\L\iota_!$ is equivalent to fully faithfulness of $\iota_*$.
    \end{proof}

  \ssec{}
    Before arriving at the main point of this section, we also need to discuss the functoriality of $\Sch$-fibred motivic spectra.
    For a morphism $f : T \to S$, there is a base change functor $\Sch_{/S} \to \Sch_{/T}$ which restricts to $\Sm_{/S} \to \Sm_{/T}$.
    If $f$ is of finite presentation, there is a left adjoint $\Sch_{/T} \to \Sch_{/S}$, the forgetful functor $(X \to T) \mapsto (X \to T \to S)$, which does \emph{not} restrict to $\Sm_{/S} \to \Sm_{/T}$ unless $f$ is smooth.
    All these functors preserve both $\A^1$-local and Nisnevich-local equivalences, and commute with the inclusions $\iota : \Sm_{/?} \to \Sch_{/?}$, so we get the following operations on the unstable categories $\uH(-)$:
      \begin{enumerate}
        \item
          For every morphism $f : T \to S$, there are direct image functors $f_* : \uH(T) \to \uH(S)$ and $f_* : \H(T) \to \H(S)$, given by restriction of presheaves along the base change functors $\Sch_{/S} \to \Sch_{/T}$ and $\Sm_{/S} \to \Sm_{/T}$, respectively.
          Moreover, we have $\iota^* f_* = f_* \iota^*$.
        
        \item
          For every morphism $f : T \to S$, there are inverse image functors $f^* : \uH(S) \to \uH(T)$ and $f^* : \H(S) \to \H(T)$, left adjoint to $f_*$.
          They are given by $(\A^1,\Nis)$-localizing the left Kan extensions along the respective base change functors.
          Moreover, we have $\L\iota_! \circ f^* = f^* \circ\L\iota_!$.

        \item
          If $f$ is of finite presentation (resp. smooth), then $f^*$ admits a left adjoint $f_\sharp : \uH(T) \to \uH(S)$ (resp. $f_\sharp : \H(T) \to \H(S)$), given by $(\A^1,\Nis)$-localizing the left Kan extension along the forgetful functor.
          In fact, in this case $f^*$ coincides with restriction along the forgetful functor $\Sch_{/T} \to \Sch_{/S}$ (resp. $\Sm_{/T} \to \Sm_{/S}$).
          Moreover, we have $\iota^* f^* = f^* \iota^*$ and $\L\iota_! \circ f_\sharp = f_\sharp \circ\L\iota_!$ in this case.
      \end{enumerate}

      Since the functors $f^*$ commute with $\bT$-suspension, there are unique colimit-preserving functors $f^* : \uSH(S) \to \uSH(T)$ and $f^* : \SH(S) \to \SH(T)$ such that $f^* \circ \Sigma^\infty_\bT \simeq \Sigma^\infty_\bT \circ f^*$.
      Their right adjoints $f_* : \uSH(T) \to \uSH(S)$ and $f_* : \SH(T) \to \SH(S)$ are the unique limit-preserving functors that commute with $\Omega^\infty_\bT$.

      If $f$ is of finite presentation (resp. smooth), then the functors $f_\sharp$ commute with $\bT$-suspension and induce unique colimit-preserving functors $f_\sharp : \uSH(T) \to \uSH(S)$ and $f_\sharp : \SH(T) \to \SH(S)$ that commute with $\Sigma^\infty_\bT$.
      By adjunction, the functors $f^*$ commute also with $\Omega^\infty_\bT$ in this case.

  \ssec{}
    Finally, we have:

    \begin{prop}\label{prop:cohomology}
      Let $\sF \in \SH(S)$ be a $\Sm$-fibred motivic spectrum over $S$.
      Then there is a canonical isomorphism
        \begin{equation*}
          \Gamma(X, \L\iota_!(\sF))
            \simeq \Maps_{\SH(S)}(\un_S, f_*f^*(\sF))
        \end{equation*}
      for every morphism $f : X \to S$ of finite presentation.
    \end{prop}
    \begin{proof}
      By \propref{prop:SH into uSH} and the various compatibilities listed above, we have canonical isomorphisms
        \begin{align*}
          \Maps_{\SH(S)}(\un_S, f_*f^*(\sF))
            &\simeq \Maps_{\SH(X)}(\un_X, f^*(\sF))\\
            &\simeq \Maps_{\uSH(X)}(\un_X, \L\iota_! f^*(\sF))\\
            &\simeq \Maps_{\uSH(X)}(\un_X, f^*\L\iota_!(\sF))\\
            &\simeq \Maps_{\uSH(X)}(f_\sharp(\un_X), \L\iota_!(\sF))\\
            &\simeq \Maps_{\uSH(S)}(\Sigma^\infty_\T(X_+), \L\iota_! (\sF))\\
            &\simeq \Gamma(X, \L\iota_! (\sF))
        \end{align*}
      as claimed.
    \end{proof}

\section{Cisinski's theorem}

  \ssec{}
    Recall the following theorem \cite[Prop.~3.7]{Cisinski}:

    \begin{thm}[Cisinski]\label{thm:Cisinski}
      Let $\sF \in \SH(S)$ be a $\Sm$-fibred motivic spectrum over $S$.
      Then for every abstract blow-up square of schemes
        \begin{equation*}
          \begin{tikzcd}
            Z' \ar{r}{k}\ar{d}{q}\ar{rd}{r}
              & S' \ar{d}{p}
            \\
            Z \ar{r}{i}
              & S,
          \end{tikzcd}
        \end{equation*}
      the induced square in $\SH(S)$
        \begin{equation*}
          \begin{tikzcd}
            \sF \ar{r}\ar{d}
              & i_*i^*(\sF)\ar{d}
            \\
            p_*p^*(\sF) \ar{r}
              & r_*r^*(\sF)
          \end{tikzcd}
        \end{equation*}
      is cartesian.
    \end{thm}

  The theorem is stated in \cite{Cisinski} for noetherian schemes of finite dimension, but the proof only uses the formalism of the six operations (namely, the proper base change and localization theorems), which have been extended to arbitrary bases by Hoyois \cite[App.~C]{Hoyois}.

  \ssec{}

    Let $\sF \in \SH(S)$ be a $\Sm$-fibred motivic spectrum.
    Let $X \in \Sch_{/S}$ with structural morphism $f : X \to S$, and suppose we have an abstract blow-up square
      \begin{equation*}
        \begin{tikzcd}
          Z' \ar{r}{k}\ar{d}{q}\ar{rd}{r}
            & X' \ar{d}{p}
          \\
          Z \ar{r}{i}
            & X.
        \end{tikzcd}
      \end{equation*}
    It follows from \thmref{thm:Cisinski} that we have an induced cartesian square
      \begin{equation*}
        \begin{tikzcd}
            f_*f^*(\sF) \ar{r}\ar{d}
              & f_*i_*i^*f^*(\sF)\ar{d}
            \\
            f_*p_*p^*f_*(\sF) \ar{r}
              & f_*r_*r^*f_*(\sF)
        \end{tikzcd}
      \end{equation*}
    in $\SH(S)$.
    Applying the left-exact functor $\Maps_{\SH(S)}(\un_S, -)$ and using the identifications of \propref{prop:cohomology}, we get a cartesian square
      \begin{equation*}
        \begin{tikzcd}
          \Gamma(X, \L\iota_!(\sF)) \ar{r}\ar{d}
            & \Gamma(Z, \L\iota_!(\sF)) \ar{d}
          \\
          \Gamma(X', \L\iota_!(\sF)) \ar{r}
            & \Gamma(Z', \L\iota_!(\sF)).
        \end{tikzcd}
      \end{equation*}
    Since $\sF$ was arbitrary, we have just shown:

    \begin{cor}\label{cor:Li_!(F) cdh-local}
      For every $\Sm$-fibred motivic spectrum $\sF \in \SH(S)$, the $\Sch$-fibred spaces
        \begin{equation*}
          \Omega^\infty_\T \L\iota_! \Sigma^n_\T(\sF) \simeq \Omega^\infty_\T \Sigma^n_\T \L\iota_!(\sF)
        \end{equation*}
      send abstract blow-up squares to cartesian squares, for all $n\ge0$.
    \end{cor}

\section{Cdh-local motivic spectra}

  \ssec{}
    A $\Sch$-fibred motivic spectrum is \emph{cdh-local} if it is Nisnevich-local and moreover satisfies \v{C}ech descent with respect to the topology generated by abstract blow-up squares in $\Sch_{/S}$.
    Like for the Nisnevich topology, Voevodsky's theorem \cite[Thm.~3.2.5]{AsokHoyoisWendt} implies that it is equivalent to require that every abstract blow-up square $Q$ is sent to a cartesian square of spaces $\Gamma(Q, \sF)$.

    We write $\uH_\cdh(S)$ for the \inftyCat of $\A^1$-local cdh-local $\Sch$-fibred spaces, by definition a full subcategory of $\uH(S)$.
    We write $\uSH_\cdh(S)$ for the \inftyCat of $\A^1$-local cdh-local $\Sch$-fibred spectra, by definition a full subcategory of $\uSH(S)$.
    That is, its objects are $\T$-spectra of pointed $\Sch$-fibred spaces that are $\A^1$-invariant and cdh-local.

  \ssec{}
    Consider the restriction of the cdh topology to $\Sm_{/S}$.
    Since cdh coverings of a smooth scheme are not necessarily smooth themselves, the inclusion $\iota : \Sm_{/S} \hookrightarrow \Sch_{/S}$ will fail to be cocontinuous for the cdh topology.
    In particular, the analogue of \propref{prop:SH into uSH} for cdh-local spaces does not hold.
    Instead, we can define $\SH_\cdh(S)$ differently so that \propref{prop:SH into uSH} holds tautologically.
    Namely, simply take $\SH_\cdh(S)$ to be the full subcategory of $\uSH_\cdh(S)$ generated under small colimits by objects of the form $\Omega^n_\T \Sigma^\infty_\T(X_+)$, for $X \in \Sm_{/S}$ and $n\ge0$.

  \ssec{}
    Consider the canonical functor
      \begin{equation}\label{eq:SH into uSH_cdh}
        \L_{\A^1,\cdh}\iota_! : \SH(S) \to \uSH_\cdh(S).
      \end{equation}
    We now state the main result, a reformulation of \thmref{thm:Cisinski}:

    \begin{thm}[Cisinski]\label{thm:main}
      For every qcqs scheme $S$, the canonical functor \eqref{eq:SH into uSH_cdh} is fully faithful, and induces an equivalence
        \begin{equation*}
          \SH(S) \to \SH_\cdh(S).
        \end{equation*}
    \end{thm}

    \begin{proof}
      \corref{cor:Li_!(F) cdh-local} says that, for every $\sF \in \SH(S)$, the $\Sch$-fibred motivic spectrum $\L_{\A^1,\Nis}\iota_!(\sF)$ is cdh-local.
      In other words, we have $\L_{\A^1,\Nis}\iota_!(\sF) \simeq \L_{\A^1,\cdh}\iota_!(\sF)$ for all $\sF$.
      Thus the fully faithfulness of \eqref{eq:SH into uSH_cdh} follows from \propref{prop:SH into uSH}.
      Its essential image is exactly $\SH_\cdh(S)$ by construction.
    \end{proof}


\bibliographystyle{halphanum}

\noindent
Institute of Mathematics, Academia Sinica, Taipei, 10617, Taiwan

\end{document}